\documentclass{article}
\usepackage{amsthm,amsmath,amssymb,tikz-cd}
\usepackage{hyperref}
\usepackage{authblk}

\usepackage{biblatex}
\newtheorem{theorem}{Theorem}
\newtheorem{corollary}{Corollary}
\newtheorem{lemma}{Lemma}
\newtheorem{proposition}{Proposition}

\theoremstyle{definition}
\newtheorem{definition}{Definition}

\theoremstyle{remark}
\newtheorem{remark}{Remark}
\newtheorem{example}{Example}

\title{A topological reading of inductive and coinductive definitions in Dependent Type Theory}
\author{Pietro Sabelli}
\affil{University of Padua\\Dipartimento di Matematica ``Tullio Levi-Civita''}

\begin{document}

\maketitle

\begin{abstract}
In the context of dependent type theory, we show that coinductive predicates have an equivalent topological counterpart in terms of coinductively generated positivity relations, introduced by G. Sambin to represent closed subsets in point-free topology. Our work is complementary to a previous one with M.E. Maietti, where we showed that, in dependent type theory, the well-known concept of wellfounded trees has a topological equivalent counterpart in terms of proof-relevant inductively generated formal covers used to provide a predicative and constructive representation of complete suplattices. All proofs in Martin-Löf's type theory are formalised in the Agda proof assistant.
\end{abstract}

\section{Introduction}
Our work aims to provide a topological presentation of coinductive predicates in dependent type theory by using tools of point-free topology developed in the field of Formal Topology.

Formal Topology is the study of topology in a constructive and predicative setting (see for example \cite{somepoints}). Its main object of investigation is a point-free notion of topological space, called \textit{basic topology}, whose definition avoids impredicative uses of the power constructor. In particular, a basic topology consists of two relations, called \textit{basic cover} and \textit{positivity relation}, primitively representing its open and closed subsets, respectively. Powerful techniques for inductively generating basic covers and coinductively generating positivity relations have been developed in \cite{CSSV03, dynf} and have since been a cornerstone of the field. 

In previous work with M.E. Maietti (\cite{topcount}), we compared the inductive generation of basic covers with other established inductive schemes in dependent type theory, proving an equivalence between inductive basic covers and \textit{inductive predicates} – conceived in the context of axiomatic set theories \cite{aczelind,rathjen} and adapted by us in the Minimalist Foundation; and between inductive basic covers and the type of well-founded trees, known as $\mathsf{W}$-types, in Martin-Löf's type theory assuming either $\eta$-equalities or function extensionality.

In this work, we extend the comparison to their coinductive counterparts; namely, we relate the coinductive generation of positivity relations with \textit{coinductive predicates} and non-wellfounded trees, known as $\mathsf{M}$-types, defined in various extensions of Martin-Löf's type theory, most notably in Homotopy Type Theory \cite{m-types}. In particular, we prove the equivalence of coinductive predicates and coinductive positivity relations in the Minimalist Foundation and, in turn, the definability of their proof-relevant versions as particular $\mathsf{M}$-types in Martin-Löf's type theory.

Since we were required to compare concepts taken from different foundational systems for mathematics, we extensively relied on the Minimalist Foundation as a ground theory to express those various notions. The Minimalist Foundation, conceived in \cite{mtt} and finalised in \cite{m09}, is a foundational system precisely designed to constitute a common core between all the most relevant foundations for constructive mathematics. This means that definitions, theorems and proofs written in its language can be exported soundly in a foundation of one's choice.


All the proofs performed within Martin-Löf's type theory have been checked in Agda. The source code is available on \\ \href{https://github.com/PietroSabelli/topological-co-induction}{https://github.com/PietroSabelli/topological-co-induction}.

\paragraph*{Structure of the paper.}
In Section \ref{prelim}, we precise the dependent type theories in which we work and lay down some notation that we will exploit throughout the paper. In Section \ref{mfsec}, working in the Minimalist Foundation, we introduce (co)inductive predicates and prove their equivalence with the (co)inductive methods of Formal Topology. In Section \ref{hottsec}, working in the proof-relevant environment of Martin-Löf's type theory, we generalise the results on the constructors of the previous section, relating them with $\mathsf{W}$-types and $\mathsf{M}$-types. Finally, in Section \ref{compsec}, we draw some conclusions about the compatibility of the Minimalist Foundation extended with (co)inductive constructors.

\section{Preliminaries}\label{prelim}

In this paper, we will mainly work within the Minimalist Foundation and Martin-Löf's type theory. We should briefly recall them.

\paragraph*{Minimalist Foundation.} The Minimalist Foundation is a two-level foundation which consists of an extensional level $\mathbf{eMF}$ (called $\mathbf{emTT}$ in \cite{m09}), understood as the actual foundation in which constructive mathematics is formalised and developed, and an intensional level $\mathbf{iMF}$ (called $\mathbf{mTT}$ in \cite{mtt} and in \cite{m09}), which acts as a programming language enjoying a realizability interpretation á la Kleene (see \cite{MMM, IMMS}). Both levels are formulated as dependent type theories à la Martin-Löf, and they are linked by a setoid interpretation, which implements proofs of the extensional level as proofs/programs of the intensional one.

In both $\mathbf{eMF}$ and $\mathbf{iMF}$, there are four kinds of types: small propositions, propositions, sets and collections (denoted respectively $\mathsf{prop_s}$, $\mathsf{prop}$, $\mathsf{set}$ and $\mathsf{col}$). Sets are particular collections, just as small propositions are particular propositions. Moreover, we identify a proposition (respectively, a small proposition) with the collection (respectively, the set) of its proofs. Eventually, we have the following square of inclusions between kinds.
\[
\begin{tikzcd}
\mathsf{prop} \arrow[r, hook]                   & \mathsf{col}                 \\
\mathsf{prop_s} \arrow[r, hook] \arrow[u, hook] & \mathsf{set} \arrow[u, hook]
\end{tikzcd}
\]
This fourfold distinction allows one to differentiate, on the one hand, between logical and mathematical entities. On the other, between different degrees of complexity (corresponding to the usual distinction between sets and classes in set theory – or the one between small and large types in Martin-Löf's type theory with a universe), thus guaranteeing the predicativity of the theory.

Propositions are those of predicate logic with equality; a proposition is small if all its quantifiers and equality predicates are restricted to sets. The base sets are just the empty set $\mathsf{N}_0$ and the singleton set $\mathsf{N}_1$; the set constructors are the dependent sum $\Sigma$, the dependent product $\Pi$, the disjoint sum $+$, the list constructor $\mathsf{List}$, and, only in the extensional level, a constructor $-/-$ for quotient a set by a small equivalence relation. The only base (proper) collection is, for the intensional level, a universe á la Russell of small propositions $\mathsf{Prop_s}$, and, for the extensional level, a classifier $\mathcal{P}(\mathsf{N_1})$ of small propositions \textit{up to equiprovability}, which, as its notation suggests, can be thought of as the collection of singleton's subsets; collections are closed only under dependent sums and under function spaces of sets towards $\mathsf{Prop_s}$ or $\mathcal{P}(\mathsf{N_1})$, depending on the level.

In the intensional level, propositions are proof-relevant, the equality predicate is intensional, and the only computation rules are $\beta$-equalities. While in the extensional level, propositions are proof-irrelevant, the equality predicate is extensional, and also all $\eta$-equalities are valid.

A crucial remark is the fact that, in the Minimalist Foundation, elimination rules of propositional constructors act only towards propositions; in this way, neither the axiom of choice nor of unique choice is validated since, in general, the system cannot internally extract a witness from a proof an existential statements \cite{choice}.

\paragraph*{Martin-Löf's type theory.} In this work, we consider a version $\mathbf{MLTT_0}$ of intensional Martin-Löf's type theory \cite{MLTT} with the following type constructors: the empty type $\mathsf{N_0}$, the unit type $\mathsf{N_1}$, dependent sums $\Sigma$, dependent products $\Pi$, the list constructor $\mathsf{List}$, identity types $\mathsf{Id}$, disjoint sums $+$, and a universe of small types $\mathsf{U_0}$ á la Russell closed under all the above type constructors. Inductive type constructors are defined to allow elimination toward all (small and large) types.

The intensionality of the theory means that judgmental equality is not reflected by propositional equality. We do instead assume $\eta$-equalities for $\mathsf{N_1}$, $\Sigma$-types, and $\Pi$-types. These weaker extensional assumptions are justified in an intensional context since they do not break any computational property, even being assumed by default in the Agda implementation of Martin-Löf's type theory.

Finally, we will also consider extending the theory with the axiom of function extensionality $\mathsf{funext}$, and the additional elimination scheme for the identity type known as \textit{axiom $\mathsf{K}$} introduced in \cite{streicher}.

\paragraph*{Compatibility.}\label{comp}
As already mentioned, the Minimalist Foundation is compatible with various foundational theories. Formally, this means that there are interpretations of the former into the latter, which preserve the meaning of logical and mathematical entities.
Moreover, when comparing the Minimalist Foundation with another foundational theory, one can choose, depending on the degree of abstraction of the latter, which of the two levels (intensional or extensional) to compare it with.

With the above provisions, we have the following compatibility results \cite{m09}.
\begin{itemize}
\item $\mathbf{iMF}$ is compatible with $\mathbf{MLTT_0}$ by identifying propositions and collections with types and small propositions and sets with small types;
\item $\mathbf{iMF}$ is compatible with the Calculus of Inductive Constructions $\mathbf{CIC}$ \cite{CoC};
\item $\mathbf{eMF}$ is compatible with constructive set theory $\mathbf{CZF}$ \cite{czf} by interpreting sets as ($\mathbf{CZF}$-)sets, collections as classes, propositions as first-order logic formulas and small propositions as bounded formulas.
\item $\mathbf{eMF}$ is compatible with the calculus of topoi $\mathcal{T}_{\mathbf{Topos}}$ as defined in \cite{modular} by identifying propositions with mono-types;
\item both $\mathbf{iMF}$ and $\mathbf{eMF}$ are compatible with Homotopy Type Theory by interpreting collections as h-sets, sets as h-sets in the first universe, propositions as h-propositions, small propositions as h-propositions in the first universe, and judgmental equalities of $\mathbf{eMF}$ as canonical propositional equalities of Homotopy Type Theory \cite{MFhott}.
\end{itemize}

\paragraph*{Encodings.}
The following notion formally says when a theory is expressive enough to define a given constructor.

If $\mathcal{T}$ is a dependent type theory and $\mathsf{C}$ is a type constructor – intended as the sets of its rules defined in the language of $\mathcal{T}$ – we say that $\mathcal{T}$ \textit{encodes} $\mathsf{C}$ if each new symbol appearing in $\mathsf{C}$ can be interpreted in $\mathcal{T}$ in such a way that all the rules of $\mathsf{C}$ are valid under this interpretation. 
Given two type constructors $\mathsf{C}$ and $\mathsf{D}$, we say that they are \textit{mutually encodable over} $\mathcal{T}$ if $\mathcal{T}+\mathsf{C}$ encodes $\mathsf{D}$ and $\mathcal{T}+\mathsf{D}$ encodes $\mathsf{C}$.

As an example, consider the set of natural numbers $\mathbb{N}$ defined as an inductive type in the usual way; it is encodable in $\mathbf{eMF}$ using the following interpretation.
\begin{align*}
\mathbb{N} & :\equiv \mathsf{List}(\mathsf{N_1}) \\
0 & :\equiv \epsilon \\
\mathsf{succ}(n) & :\equiv \mathsf{cons}(n,\star) \\
\mathsf{El}_\mathbb{N}(n,b,(x,z).c) & :\equiv \mathsf{El}_\mathsf{List}(n,b,(x,y,z).c)
\end{align*}
\\

\paragraph*{Notation.}
We will use the following conventions and shorthands to improve readability:
\begin{itemize}
\item when writing inference rules, the piece of context common to all the judgments is omitted;
\item we reserve the arrow symbol $\to$ \textit{(resp. $\times$)} as a shorthand for a non-dependent function space \textit{(resp. for non-dependent product sets)}, and we denote the implication connective with the arrow symbol $\Rightarrow$;
\item if $a \in A$ and $f \in (\Pi x \in A)B(x)$, we will often write $f(a)$ as a shorthand for $\mathsf{Ap}(f,a)$; analogously for $\mathsf{Ap}_\forall(f,a)$ and $\mathsf{Ap}_\Rightarrow(f,a)$; when lambda abstracting, we omit the domain and the kind of lambda abstraction $\lambda$, $\lambda_\forall$, $\lambda_\Rightarrow$; analogously, we omit the kind of pair constructors $\langle-,- \rangle$, $\langle-,_\wedge- \rangle$ and $\langle-,_\exists- \rangle$.
\item we define the two projections from a dependent sum in the usual way $\pi_{\mathsf{1}}(z) :\equiv \mathsf{El}_{\Sigma}(z,(x,y).x)$ and
$\pi_{\mathsf{2}}(z) :\equiv \mathsf{El}_{\Sigma}(z,(x,y).y)$;
\item negation, the true constant, and logical equivalence are defined in the usual way
$\neg \varphi :\equiv \varphi \Rightarrow \bot$,
$\top :\equiv \neg\bot$ and $\varphi \Leftrightarrow \psi :\equiv \varphi \Rightarrow \psi \wedge \psi \Rightarrow \varphi$;
\item we will sometimes render the judgment $\mathsf{true} \in \varphi \; [\Gamma]$ as $\varphi \; \mathsf{true} \; [\Gamma]$;
\item as usual in $\mathbf{eMF}$, we let $\mathcal{P}(A) :\equiv A \to \mathcal{P}(\mathsf{N_1})$ to denote the power-collection of a set $A$, and $a \,\varepsilon\, V :\equiv \mathsf{Ap}(V,a)$ to denote the (propositional) relation of membership between terms $a \in A$ and subsets $V \in \mathcal{P}(A)$; moreover, given $\varphi(x) \; \mathsf{prop} \; [x \in A]$ we will use the following set-theoretic shorthands
\begin{align*}
\{ x \in A \,|\, \varphi(x) \}& :\equiv \lambda x.[\varphi] \in \mathcal{P}(A)\\
\emptyset & :\equiv \{ x \in A \,|\, \bot \} \in \mathcal{P}(A) \\
A & :\equiv  \{ x \in A \,|\, \top \} \in \mathcal{P}(A) \\
(\forall x \,\varepsilon\, V)\varphi(x) & :\equiv (\forall x \in A)(x \,\varepsilon\, V \Rightarrow \varphi(x)) \; \mathsf{prop} \\
(\exists x \,\varepsilon\, V)\varphi(x) & :\equiv (\exists x \in A)(x \,\varepsilon\, V \,\wedge\; \varphi(x)) \; \mathsf{prop}
\end{align*}
\\
\end{itemize}

\section{(Co)Induction in the Minimalist Foundation}\label{mfsec}
The basic idea behind any (co)inductive construction is to generate an object by proceeding from a given set of rules – intuitively, with induction following them forward and coinduction backwards. How such rules are specified and used heavily depends on \textit{(1)} the kind of mathematical object to be constructed, and \textit{(2)} the setting in which this construction is formalised.

Concerning the first point, a key distinction is between the (co)inductive generation of sets and the (co)inductive generation of predicates (or, equivalently, subsets). The present section focuses on the latter. Paradigmatic examples of inductively and coinductively defined sets are lists (of which natural numbers are the most basic instance) and streams, respectively. On the other hand, the fundamental example of (co)inductive predicates we would consider is given by deductive systems, which, through their rules, for a given set of formulae declare inductively which are the theorems and coinductively which are the refutables.

Concerning the second point, each foundational theory has its peculiar ways of implementing (co)induction, e.g. (non-)wellfounded trees in Martin-Löf's type theories, Higher Inductive Types in Homotopy Type Theory, or Generalized Inductive Definitions in set theory. We are interested in adapting the latter scheme, first introduced in \cite{aczelind} and then adapted in a constructive setting in \cite{rathjen, czf}, to dependent type theory. To this end, in this section, we define it in the Minimalist Foundation, first at the extensional level, then at the intensional one. Then, in the next section, we will show how to define its proof-relevant version in Martin-Löf's type theory.

\subsection{(Co)Inductive predicates in the extensional level} 

The starting point is to specify how to declare rules for the (co)inductive generation of predicates. We do so using a notion that goes under two names, corresponding to two different – although related – interpretations. The first name is \textit{rule set}\footnote{In \cite{rathjen}, the name is changed to \textit{inductive definition}; we stick to the original terminology since it helps intuition, and our treatment is not limited to induction.}, and, in this sense, it is a straightforward adaptation in the setting of the Minimalist Foundation of the homonym notion in \cite{aczelind}. The second name is \textit{axiom set}, defined in the context of (co)inductive generation of formal topologies in \cite{CSSV03}. We would focus on the first interpretation (and hence use the first name) while presenting (co)inductive predicates. We will turn to the second name when recalling (co)inductive methods in formal topology.
\begin{definition}
A \textnormal{rule} or \textnormal{axiom set} over a given set $A$ consists of the following data:
\begin{enumerate}
\item a dependent family of sets $I(x) \; \mathsf{set} \; [x \in A]$;
\item a dependent family of $A$-subsets $C(x,y) \in \mathcal{P}(A) \; [x \in A,y \in I(x)]$.
\end{enumerate}
Given two elements $a \in A$ and $i \in I(a)$, we say that $i$ is a \textit{rule} with \textit{premises} $C(a,b)$ and \textit{conclusion} $a$; sometimes it is represented pictorially as
\[
\frac{C(a,i)}{a}\;i
\]
\end{definition}
Hopefully, the chosen terminology makes the logical interpretation of a rule set as an internally defined deduction system transparent. For the rest of this section, suppose to have fixed a set $A$ and a rule set $(I,C)$ over it.

Recall that a predicate $P$ on $A$ is just a proposition depending on $A$
\[
P(x) \; \mathsf{prop} \; [x \in A]
\]
A rule set induces two ways of transforming predicates.
\begin{definition}
Given a predicate $P$ on $A$, we define other two predicates $\mathsf{Der}_{I,C}(P)$ and $\mathsf{Conf}_{I,C}(P)$ on $A$ as follows:
\begin{align*}
\mathsf{Der}_{I,C}(P)(x) & :\equiv (\exists y \in I(x))(\forall z \,\varepsilon\, C(x,y))P(z) \; \mathsf{prop} \; [x \in A]\\
\mathsf{Conf}_{I,C}(P)(x) & :\equiv (\forall y \in I(x))(\exists z \,\varepsilon\, C(x,y))P(z) \; \mathsf{prop} \; [x \in A]
\end{align*}
We call them \textnormal{derivability} and \textnormal{confutability from} $P$, respectively.
\end{definition}
As the names suggest, derivability from $P$ tells which elements of $A$ can be derived with exactly one rule application, assuming as axioms the elements for which $P$ holds – in fact, the definition of $\mathsf{Der}_{I,C}(P)(x)$ explicitly reads \textit{there exists a rule with conclusion $x$ such that all its premises are satisfied by $P$}. Dually, confutability from $P$ tells which elements of $A$ can be confuted after exactly one step of backward search, assuming that the elements for which $P$ holds are already refuted – the definition of $\mathsf{Conf}_{I,C}(P)(x)$ reads \textit{all rules with conclusion $x$ have at least one premise for which $P$ holds}.

In accordance with the above interpretation, the constructions $\mathsf{Der}_{I,C}(-)$ and $\mathsf{Conf}_{I,C}(-)$ can be read meta-theoretically as two endomorphisms of the preorder of predicates on $A$: the preorder relation is formally given by
\[
P \leq_A Q :\equiv (\forall x \in A)(P(x) \Rightarrow Q(x)) \; \mathsf{prop} 
\]
and it is straightforward to check that $P \leq_A Q$ implies both \[
\mathsf{Der}_{I,C}(P) \leq_A \mathsf{Der}_{I,C}(Q) \quad \text{and}\quad \mathsf{Conf}_{I,C}(P) \leq_A \mathsf{Conf}_{I,C}(Q)\text{.}\]
Finally, notice that if $P$ is a small proposition, then also $\mathsf{Der}_{I,C}(P)$ and $\mathsf{Conf}_{I,C}(P)$ are. Thanks to these observations, we can exploit the theory and terminology of monotone operators.
\begin{definition}
We say that a predicate $P$ on $A$ is \textnormal{closed} with respect to the rule set $(I,C)$, if it is closed with respect to $\mathsf{Der}_{I,C}(-)$, namely if
\[
\mathsf{Der}_{I,C}(P) \leq_A P \; \mathsf{true}
\]
Dually, we say that $P$ is \textnormal{consistent} with respect to the rule set $(I,C)$, if it is consistent with respect to $\mathsf{Conf}_{I,C}(-)$, namely if
\[
P \leq_A \mathsf{Conf}_{I,C}(P) \; \mathsf{true}
\]
\end{definition}
Observe that, thanks to monotonicity, if $P$ is a closed predicate, we have the following chain of inequalities.
\[
\cdots \leq_A  \mathsf{Der}_{I,C}(\mathsf{Der}_{I,C}(P)) \leq_A \mathsf{Der}_{I,C}(P) \leq_A P
\]
Dually, if $P$ is consistent it, follows
\[
P \leq_A \mathsf{Conf}_{I,C}(P) \leq_A \mathsf{Conf}_{I,C}(\mathsf{Conf}_{I,C}(P))   \leq_A \cdots
\]
Therefore, to be a closed predicate means that the deductive system cannot derive anything new from it. Dually, to be a consistent predicate means that the elements for which it holds, if assumed refuted by the deduction system, will still be after any number of backward search steps. We are then naturally led to interpret the smallest closed predicate as expressing derivability and the greatest consistent predicate as confutability. In the base theory, their existence is not guaranteed, let alone their smallness as propositions. Therefore, we should postulate them. Formally, this is done by introducing two new logical constructors $\mathsf{Ind}$ and $\mathsf{CoInd}$. We now report their precise rules, together with the judgements that formalise the rule set; the latter will then be left implicit in the premises of the former.

\paragraph*{Rule set parameters in eMF}
\[
A\; \mathsf{set} \quad I(x) \; \mathsf{set} \; [x \in A] \quad C(x,y) \in \mathcal{P}(A) \; [x \in A, y \in I(x)] 
\]

\paragraph*{Rules for inductive predicates in eMF}
\[
\textsf{F-$\mathsf{Ind}$}\;\frac{a \in A}{\mathsf{Ind}_{I,C}(a) \; \mathsf{prop_s}}
\]
\[
\textsf{I-$\mathsf{Ind}$}\;\frac{}{  \mathsf{Der}_{I,C}(\mathsf{Ind}_{I,C}) \leq_A \mathsf{Ind}_{I,C} \; \mathsf{true}}
\]
\[
\textsf{E-$\mathsf{Ind}$}\;\frac{P(x) \; \mathsf{prop} \; [x \in A]\quad \mathsf{Der}_{I,C}(P) \leq_A P \; \mathsf{true}}{\mathsf{Ind}_{I,C} \leq_A P \; \mathsf{true}}
\]

\paragraph*{Rules for coinductive predicates in eMF}
\[
\textsf{F-$\mathsf{CoInd}$}\;\frac{a \in A}{\mathsf{CoInd}_{I,C}(a) \; \mathsf{prop_s}}
\]
\[
\textsf{E-$\mathsf{CoInd}$}\;\frac{}{ \mathsf{CoInd}_{I,C} \leq_A \mathsf{Conf}_{I,C}(\mathsf{CoInd}_{I,C})  \; \mathsf{true}}
\]
\[
\textsf{I-$\mathsf{CoInd}$}\;\frac{P(x) \; \mathsf{prop} \; [x \in A]\quad P \leq_A \mathsf{Conf}_{I,C}(P) \; \mathsf{true}}{P \leq_A \mathsf{CoInd}_{I,C} \; \mathsf{true}}
\]

Notice how, again, consistently with the intended interpretation, we can deduce from their rules that $\mathsf{Ind}_{I,C}$ and $\mathsf{CoInd}_{I,C}$ are fixed points for $\mathsf{Der}_{I,C}$ and $\mathsf{Conf}_{I,C}$, respectively.

\begin{remark}\label{imp}
Classically, i.e. assuming the principle of excluded middle, inductive and coinductive predicates are complementary subsets and thus mutually encodable as
\begin{align*}
\mathsf{CoInd}_{I,C}(x) & :\equiv \neg\mathsf{Ind}_{I,C}(x) \; \mathsf{prop} \; [x \in A] \\
\mathsf{Ind}_{I,C}(x) & :\equiv \neg\mathsf{CoInd}_{I,C}(x) \; \mathsf{prop} \; [x \in A]
\end{align*}
To see it, observe that the following equivalences between predicates hold.
\begin{align*}
\neg\mathsf{Der}_{I,C}(P)(x) & \Leftrightarrow \mathsf{Conf}_{I,C}(\neg P)(x) \; \mathsf{true} \; [x \in A] \\
\neg\mathsf{Conf}_{I,C}(P)(x) & \Leftrightarrow \mathsf{Der}_{I,C}(\neg P)(x) \;\,\mathsf{true} \; [x \in A]
\end{align*}
Therefore, a predicate $P$ is closed (resp. consistent) if and only if its complement $\neg P$ is consistent (resp. closed) – and the smallest closed predicate is the complement of the greatest consistent one, and vice versa.

On the other hand, both inductive and coinductive predicates are impredicatively encodable in the calculus $\mathcal{T}_{\mathbf{Topos}}$. This is done in the usual way, interpreting them as the intersection of all closed predicates and the union of all consistent predicates, respectively.
\begin{align*}
\mathsf{Ind}_{I,C}(x) &:\equiv (\forall P \in \mathcal{P}(A))(\mathsf{Der}_{I,C}(P) \leq_A P \Rightarrow x\,\varepsilon\,P)  \; \mathsf{prop} \; [x \in A] \\
\mathsf{CoInd}_{I,C}(x) &:\equiv (\exists P \in \mathcal{P}(A))(P \leq_A \mathsf{Conf}_{I,C}(P) \wedge x\,\varepsilon\,P)  \; \mathsf{prop} \; [x \in A]
\end{align*}
\end{remark}
We end the presentation of (co)inductive predicates in the extensional level by presenting alternative introduction and elimination rules for them, in which some internal quantifiers are externalised. Although its mathematical content is perhaps less compact, they have the advantage, on the one hand, of emphasising the characters of introduction and elimination, thus being more easily suited to design their intensional counterparts, and, on the other hand, of being more similar to the rules for the (co)inductive generation of formal topologies, and thus being more easily comparable with them.
\begin{lemma}\label{altex}
The introduction and elimination rules for inductive predicates in $\mathbf{eMF}$ can equivalently be formulated in the following way.
\[
\mathsf{I}\textsf{-}\mathsf{Ind}'\;\frac{a \in A \quad i \in I(a) \quad (\forall x \,\varepsilon\, C(a,i))\mathsf{Ind}_{I,C}(x) \; \mathsf{true}}{\mathsf{Ind}_{I,C}(a) \; \mathsf{true}}
\]
\[
\mathsf{E}\textsf{-}\mathsf{Ind}'\;\frac{
\begin{aligned}
& P(x) \; \mathsf{prop} \; [x \in A] \\
& P(x) \; \mathsf{true} \; [x \in A, y \in I(x), w \in (\forall z \,\varepsilon\, C(x,y))P(z)] \\
& a \in A \quad \mathsf{Ind}_{I,C}(a) \; \mathsf{true}
\end{aligned}}
{P(a) \; \; \mathsf{true}}
\]
Analogously, for coinductive predicates.
\[
\mathsf{E}\textsf{-}\mathsf{CoInd}'\;\frac{a \in A \quad i \in I(a) \quad \mathsf{CoInd}_{I,C}(a) \; \mathsf{true}}{(\exists x \,\varepsilon\, C(a,i))\mathsf{CoInd}_{I,C}(x) \; \mathsf{true}}
\]
\[
\mathsf{I}\textsf{-}\mathsf{CoInd}'\;\frac{
\begin{aligned}
& P(x) \; \mathsf{prop} \; [x \in A] \\
& (\exists z \,\varepsilon\, C(x,y))P(z) \; \mathsf{true} \; [x \in A, y \in I(x), w \in P(x)] \\
& a \in A \quad P(a) \; \mathsf{true}
\end{aligned}}{ \mathsf{CoInd}_{I,C}(a) \; \mathsf{true}}
\]
\end{lemma}
\begin{proof}
We show how to prove that $\textsf{I-Ind}$ and $\textsf{I-Ind}'$ are equivalent, the other equivalences being proven analogously.

First, notice that the rule $\textsf{I-Ind}$ explicitly reads
\begin{equation}\label{eq:1}
(\forall x \in A)((\exists y \in I(x))(\forall z \,\varepsilon\, C(x,y))\mathsf{Ind}_{I,C}(z) \Rightarrow \mathsf{Ind}_{I,C}(x)) \; \mathsf{true}   
\end{equation}
First, let us assume that \ref{eq:1} and the premises of $\textsf{I-Ind}'$ hold. To derive the conclusion of $\textsf{I-Ind}'$, apply the rule $\textsf{E-}\forall$ to the universal statement of \ref{eq:1} and on the premise $a \in A$; then apply $\textsf{E-}\Rightarrow$ to the resulting conditional statement and to a proof of the antecedent obtained with the rule $\textsf{I-}\exists$ applied to the premisses $i \in I(a)$ and $\mathsf{true} \in (\forall x \,\varepsilon\, C(a,i))\mathsf{Ind}_{I,C}(x)$.

On the other hand, suppose that the rule $\textsf{I-Ind}'$ holds; it obviously derives the judgment
\[
\mathsf{Ind}_{I,C}(x) \; \mathsf{true} \; [x \in A,y \in I(x), p \in (\forall z \,\varepsilon\, C(x,y))\mathsf{Ind}_{I,C}(z)]
\]
One then proceeds by applying successively the rules $\textsf{E-}\exists$, $\textsf{I-}\Rightarrow$, and $\textsf{I-}\forall$ to obtain \ref{eq:1}.
\end{proof}

\subsection{(Co)Inductive predicates in the intensional level}
We now introduce the intensional counterparts of the rules for inductive and coinductive predicates. They are crucial, on the one hand, to extend the Minimalist Foundation's intensional level and have it interpret extensional (co)induction in its quotient model; on the other hand, to be able to compare our notion of (co)induction with that of other intensional theories.

A rule set is formalised similarly, except for the usual replacement of subsets for functions towards the universe of small propositions.
\paragraph*{Rule set parameters in iMF}
\[
A\; \mathsf{set} \quad I(x) \; \mathsf{set} \; [x \in A] \quad C(x,y) \in A \to \mathsf{Prop_s} \; [x \in A, y \in I(x)] 
\]

The rules are then easily obtained from the extensional ones presented as in Lemma \ref{altex} by adding the now-relevant proof terms. While we can give computational meaning to inductive proofs drawing on the usual pattern of inductive types in Martin-Löf's type theory, the computational meaning of coinduction is still an open problem; therefore, its proof terms will be given axiomatically (which, of course, will block computation). Again, we keep implicit in the premises the rule set parameters.
\paragraph*{Rules for inductive predicates in iMF}
\[
\textsf{F-$\mathsf{Ind}$}\;\frac{a \in A}{\mathsf{Ind}_{I,C}(a) \; \mathsf{prop_s}}
\]
\[
\textsf{I-$\mathsf{Ind}$}\;
\frac
{
a \in A \quad i \in I(a) \quad p \in (\forall x \,\varepsilon\, C(a,i))\mathsf{Ind}_{I,C}(x)
}
{
\mathsf{ind}(a,i,p) \in \mathsf{Ind}_{I,C}(a)
}
\]
\[
\textsf{E-$\mathsf{Ind}$}\;\frac{
\begin{aligned}
& P(x) \; \mathsf{prop} \; [x \in A] \\
& c(x,y,w) \in P(x) \; [x \in A, y \in I(x), w \in (\forall z \,\varepsilon\, C(x,y))P(z)] \\
& a \in A \quad p \in \mathsf{Ind}_{I,C}(a)
\end{aligned}}
{\mathsf{El_{Ind}}(a,p,(x,y,w).c) \in P(a)}
\]
\[
\textsf{C-$\mathsf{Ind}$}\;\frac{
\begin{aligned}
& P(x) \; \mathsf{prop} \; [x \in A] \\
& c(x,y,w) \in P(x) \; [x \in A, y \in I(x), w \in (\forall z \,\varepsilon\, C(a,i))P(z)] \\
& a \in A \quad i \in I(a) \quad p \in (\forall x \,\varepsilon\, C(a,i))\mathsf{Ind}_{I,C}(x)
\end{aligned}}
{\mathsf{El_{Ind}}(a,\mathsf{ind}(a,i,p),(x,y,w).c) = c(a,i,\lambda z.\lambda q.\mathsf{El_{Ind}}(z,p(z,q),(x,y,w).c)) \in P(a)}
\]

\paragraph*{Rules for coinductive predicates in iMF}
\[
\textsf{F-$\mathsf{CoInd}$}\;\frac{a \in A}{\mathsf{CoInd}_{I,C}(a) \; \mathsf{prop_s}}
\]
\[
\textsf{E-$\mathsf{CoInd}$}\;\frac{a \in A \quad i \in I(a) \quad p \in \mathsf{CoInd}_{I,C}(a)}{ \mathsf{des}(a,i,p) \in (\exists x \,\varepsilon\, C(a,i))\mathsf{CoInd}_{I,C}(x)}
\]
\[
\textsf{I-$\mathsf{CoInd}$}\;\frac{
\begin{aligned}
& P(x) \; \mathsf{prop} \; [x \in A] \\
& c(x,y,w) \in (\exists z \,\varepsilon\, C(x,y))P(z) \; [x \in A, y \in I(x), w \in P(x)] \\
& a \in A \quad p \in P(a)
\end{aligned}}{\mathsf{coind}(a,p,(x,y,w).c) \in \mathsf{CoInd}_{I,C}(a)}
\]

\begin{proposition}\label{iimp}
Inductive and coinductive predicates are encodable in the Calculus of Constructions.
\end{proposition}
\begin{proof}
The Calculus of Constructions can be regarded as an impredicative version of $\mathbf{iMF}$. Thus, the proof adapts the same idea of Remark \ref{imp}. We only need to make explicit the proof terms and, regarding the inductive predicate constructor, additionally show that the conversion rule is satisfied.
\begin{align*}
\mathsf{Ind}_{I,C}(a)   :\equiv\; & (\forall P \in A \to \mathsf{Prop})
\\
 & ((\forall x \in A)(\forall y \in I(x))((\forall z \,\varepsilon\,C(x,y))P(z)\Rightarrow P(x)) \\
& \Rightarrow P(a))
\\
\mathsf{ind}(a,i,p) :\equiv \;& \lambda P.\lambda c.c(a,i,\lambda x.\lambda q.p(x,q,P,c))
\\
\mathsf{El}^P_{\mathsf{Ind}}(a,p,(x,y,w).c) :\equiv \; & p(\lambda x . P(x),\lambda x.\lambda y.\lambda w.c(x,y,w))
\end{align*}
\begin{align*}
\mathsf{CoInd}_{I,C}(a) :\equiv\; & (\exists P \in A \to \mathsf{Prop})
\\
 & ((\forall x \in A)(P(x) \Rightarrow (\forall y \in I(x))(\exists z \,\varepsilon\,C(x,y))P(z)) \\
& \wedge P(a)) \\
\mathsf{des}(a,i,p) :\equiv\; & \mathsf{El}_\exists(p,(P,q).\mathsf{El}_\exists(c(a,r,i),(z,s).\langle z , \langle \pi_1(s) , \langle P,c,\pi_2(s) \rangle \rangle \rangle))
\\
\text{where }& c :\equiv \pi_1(q) \quad r :\equiv \pi_2(q)\\
\mathsf{coind}^P(a,p,(x,y,w).c) :\equiv\; & \langle \lambda x.P(x), \lambda x.\lambda y.\lambda w.c(x,y,w), p \rangle
\end{align*}
\end{proof}

The extension of the quotient model and the realizability interpretation for the two new constructors will follow from the results already obtained for topological (co)induction once we establish their equivalence in the next section.

\subsection{Topological (co)induction}
The development of the Minimalist Foundation has always been closely tied to that of Formal Topology, and it is no surprise that (co)induction has been first considered in that context. In two successive papers \cite{mmr21, mmr22}, the authors proposed extensions of the Minimalist Foundation supporting the formalisation of (co)inductive generation methods developed in Formal Topology. Moreover, they showed that both the setoid interpretation of the extensional level into the intensional one and the realizability interpretation of the latter can be adapted to support those extensions.

We now recall those extensions, starting with the basic definitions of Formal Topology formalised in the Minimalist Foundation.

\begin{definition}
A \textnormal{basic topology} consists of the following data:
\begin{enumerate}
\item a set $A$, whose elements are called \textnormal{basic opens};
\item a small binary relation $\vartriangleleft$, called \textnormal{basic cover}, between elements of $A$ and subsets of $A$, satisfying the following properties:
\begin{itemize}
\item \textnormal{(reflexivity)} if $a \,\varepsilon\, V$, then $a \vartriangleleft V$;
\item \textnormal{(transitivity)} if $a \vartriangleleft U$ and $(\forall x \,\varepsilon\, U)x \vartriangleleft V$, then $a \vartriangleleft V$.
\end{itemize}
\item a small binary relation $\ltimes$, called \textnormal{positivity relation}, between elements of $A$ and subsets of $A$, satisfying the following properties:
\begin{itemize}
\item \textnormal{(coreflexivity)} if $a \ltimes V$, then $a \,\varepsilon\, V$;
\item \textnormal{(cotransitivity)} if $a \ltimes U$ and  $(\forall x \in A)(x \ltimes V \Rightarrow x \,\varepsilon\, U)$, then $a \ltimes V$;
\item \textnormal{(compatibility)}
if $a \ltimes V$ and $a \vartriangleleft U$, then $(\exists x \,\varepsilon\,V )(x \ltimes U)$.
\end{itemize}
\end{enumerate}
\end{definition}

The intuitive meaning of the above definition is the following: the set $A$, as its name suggests, is a base for the topology; the relation $a \vartriangleleft V$ means that the basic open $a$ is covered by the family of basic opens $V$; and finally, $a \ltimes V$ means that there is a point of $a$ all whose basic neighbourhoods belongs to $V$.

In \cite{CSSV03}, the authors devised a way to inductively generate a basic cover on a given set $A$, starting from an axiom set $(I,C)$ over it. In this sense, an axiom set is the set-indexed family of axioms
\[
a \vartriangleleft C(a,i) \quad \text{for each $a \in A$ and $i \in I(a)$}
\]
The inductively generated basic cover is then the smallest one which satisfies them. Later, in \cite{somepoints}, this idea was dualised to coinductively generate the greatest positivity relation satisfying
\[
a \ltimes C(a,i) \quad \text{for each $a \in A$ and $i \in I(a)$}
\]
which, moreover, turns out to be compatible with the inductively generated basic cover on the same axiom set, and thus giving rise to a basic topology.

\begin{example}
Consider the set $\mathsf{List}(\mathbb{N})$, and the following axiom set over it.
\begin{align*}
I(s) & :\equiv \mathsf{N_1} + (\Sigma l \in \mathsf{List}(\mathbb{N}))(\exists t \in \mathsf{List}(\mathbb{N}))[l,t] =_{\mathsf{List}(\mathbb{N})} s \\
C(s,\mathsf{inl}(\star)) & :\equiv \{ \mathsf{cons}(s,n)  \,|\, n \in \mathbb{N} \} \\
C(s,\mathsf{inr}(z)) & :\equiv \{  \pi_1(z) \} 
\end{align*}
where $[-,-]$ is the concatenation operator.
The basic topology obtained by (co)induction from the above axiom set is the \textit{Baire space topology}. In particular, in this situation, the positivity relation $a \ltimes_{I,C} V$ precisely states that there exists a \textit{spread} containing $a$ and contained in $V$ \cite{ciraulo}.
\end{example}

When the positivity relation $a \ltimes V$ is specialised to the case where $V = A$, one talks of the \textit{positivity predicate} $\mathsf{Pos}(a) :\equiv a \ltimes A$; sometimes, coinductively generated positivity predicates have been considered on their own, as in \cite{coexp}, where the authors used them to constructively and predicatively prove the coreflection of locales in open locales.

These methods were formalised in the Minimalist Foundation in the form of an inductive constructor $\vartriangleleft$ in \cite{mmr21} and a coinductive constructor $\ltimes$ in \cite{mmr22}, having both as parameters an axiom set $(I,C)$ over a set $A$ (formalised as in the case of (co)inductive predicates), and a subset $V \in \mathcal{P}(A)$ in the extensional level, or a propositional function $V \in A \to \mathsf{Prop_s}$ in the intensional one. We first recall their rules in the extensional level. Again, we leave implicit in the premises the parameters.

\paragraph*{Rules for inductive basic covers in eMF}
\[
\textsf{F-$\vartriangleleft$}\;\frac{a \in A}{a \vartriangleleft_{I,C} V\; \mathsf{prop_s}}
\]
\[
\textsf{I\textsubscript{rf}-$\vartriangleleft$}\;\frac{a \,\varepsilon\, V\; \mathsf{true}}{a \vartriangleleft_{I,C} V\; \mathsf{true}}
\]
\[
\textsf{I\textsubscript{tr}-$\vartriangleleft$}\;\frac{a \in A \quad i \in I(a) \quad (\forall x \,\varepsilon\, C(a,i))x \vartriangleleft_{I,C} V \; \mathsf{true}}{a \vartriangleleft_{I,C} V\; \mathsf{true}}
\]
\[
\textsf{E-$\vartriangleleft$}\;\frac
{
\begin{aligned}
& P(x) \; \mathsf{prop} \; [x \in A] \\
& (\forall x \in A)(x \,\varepsilon\, V \vee (\exists y \in I(x))(\forall z \,\varepsilon\, C(x,y))P(z) \Rightarrow P(x)) \; \mathsf{true}    \\
& a \in A \quad a \vartriangleleft_{I,C} V\; \mathsf{true} 
\end{aligned}
}
{P(a) \; \mathsf{true}}
\]

\paragraph*{Rules for coinductive positivity relations in eMF}
\[
\textsf{F-$\ltimes$}\;\frac{a \in A}{a \ltimes_{I,C} V\; \mathsf{prop_s}}
\]
\[
\textsf{E\textsubscript{corf}-$\ltimes$}\;\frac{a \ltimes_{I,C} V\; \mathsf{true}}{a \,\varepsilon\, V\; \mathsf{true}}
\]
\[
\textsf{E\textsubscript{cotr}-$\ltimes$}\;\frac{a \in A \quad i \in I(a) \quad a \ltimes_{I,C} V \; \mathsf{true}}{(\exists x \,\varepsilon\, C(a,i))x \ltimes_{I,C} V\; \mathsf{true}}
\]
\[
\textsf{I-$\ltimes$}\;\frac
{
\begin{aligned}
& P(x) \; \mathsf{prop} \; [x \in A] \\
& (\forall x \in A)(P(x) \Rightarrow x \,\varepsilon\, V \wedge (\forall y \in I(x))(\exists z \,\varepsilon\, C(x,y))P(z)) \; \mathsf{true} \\
& a \in A \quad P(a) \; \mathsf{true} 
\end{aligned}
}
{a \ltimes_{I,C} V \; \mathsf{true}}
\]

The above rules are very close to that of (co)inductive predicates presented as in Lemma \ref{altex}, the only difference being the presence of the parameter $V$, which implies, for each of the two constructors, an additional (co)reflection rule and an additional condition on the predicates involved in the universal properties. To consider the clause induced by the parameter $V$, we modify the derivability and confutability endomorphisms in the following way.
\begin{align*}
\mathsf{Der}_{I,C,V}(P)(x) & :\equiv x \,\varepsilon\,V \vee \mathsf{Der}_{I,C}(P)(x) \; \mathsf{prop} \; [x \in A] \\
\mathsf{Conf}_{I,C,V}(P)(x) & :\equiv x \,\varepsilon\,V \wedge \mathsf{Conf}_{I,C}(P)(x) \; \mathsf{prop} \; [x \in A] 
\end{align*}
The next lemma shows that indeed the predicates $- \vartriangleleft_{I,C} V$ and $- \ltimes_{I,C} V$ can be seen as the smallest closed predicate of $\mathsf{Der}_{I,C,V}$ and the greatest consistent predicate of $\mathsf{Conf}_{I,C,V}$, respectively.
\begin{lemma}\label{altin}
The introduction and elimination rules for extensional inductive basic covers in $\mathbf{eMF}$ can equivalently be formulated in the following way.
\[
\mathsf{I}\textsf{-$\vartriangleleft$}'\;\frac{}{
\mathsf{Der}_{I,C,V}(- \vartriangleleft_{I,C} V) \leq_A - \vartriangleleft_{I,C} V \; \mathsf{true}}
\]
\[
\mathsf{E}\textsf{-$\vartriangleleft$}'\frac{P(x) \; \mathsf{prop} \; [x \in A]\quad \mathsf{Der}_{I,C,V}(P) \leq_A P \; \mathsf{true}}{- \vartriangleleft_{I,C} V \leq_A P \; \mathsf{true}}
\]
Analogously for extensional coinductive positivity relations.
\[
\mathsf{E}\textsf{-$\ltimes$}'\;\frac{}{
- \ltimes_{I,C} V \leq_A \mathsf{Conf}_{I,C,V}(- \ltimes_{I,C} V) \; \mathsf{true}}
\]
\[
\mathsf{I}\textsf{-$\ltimes$}'\;\frac{P(x) \; \mathsf{prop} \; [x \in A]\quad P \leq_A \mathsf{Conf}_{I,C,V}(P) \; \mathsf{true}}{ P \leq_A - \ltimes_{I,C} V \; \mathsf{true}}
\]
\end{lemma}
\begin{proof}
The proof is entirely analogous to that of Lemma \ref{altex}.
\end{proof}

The intensional rules are designed similarly. Again, with inductive basic cover enjoying a computational behaviour, while coinductive positivity relation given axiomatically.

\paragraph*{Rules for inductive basic covers in iMF}
\[
\textsf{F-$\vartriangleleft$}\;\frac{a \in A}{a \vartriangleleft_{I,C}V \; \mathsf{prop_s}}
\]
\[
\textsf{I\textsubscript{rf}-$\vartriangleleft$}\;
\frac
{
a \in A \quad r \in a\,\varepsilon\,V
}
{
\mathsf{rf}(a,r) \in a \vartriangleleft_{I,C}V
}
\]
\[
\textsf{I\textsubscript{tr}-$\vartriangleleft$}\;
\frac
{
a \in A \quad i \in I(a) \quad p \in (\forall x \,\varepsilon\, C(a,i))x \vartriangleleft_{I,C}V
}
{
\mathsf{tr}(a,i,p) \in a \vartriangleleft_{I,C}V
}
\]
\[
\textsf{E-$\vartriangleleft$}\;\frac{
\begin{aligned}
& P(x) \; \mathsf{prop} \; [x \in A] \\
& q_1(x,y) \in P(x) \; [x \in A, y \in x \,\varepsilon\,V] \\
& q_2(x,y,w) \in P(x) \; [x \in A, y \in I(x), w \in (\forall z \,\varepsilon\, C(x,y))P(z)] \\
& a \in A \quad p \in a \vartriangleleft_{I,C}V
\end{aligned}}
{\mathsf{El}_\vartriangleleft(a,p,(x,y).q_1,(x,y,w).q_2) \in P(a)}
\]
\[
\textsf{C\textsubscript{rf}-$\vartriangleleft$}\;\frac{
\begin{aligned}
& P(x) \; \mathsf{prop} \; [x \in A] \\
& q_1(x,y) \in P(x) \; [x \in A, y \in x \,\varepsilon\,V] \\
& q_2(x,y,w) \in P(x) \; [x \in A, y \in I(x), w \in (\forall z \,\varepsilon\, C(x,y))P(z)] \\
& a \in A \quad r \in a \,\varepsilon\,V
\end{aligned}}
{\mathsf{El}_\vartriangleleft(a,\mathsf{rf}(a,r),(x,y).q_1,(x,y,w).q_2) = q_1(a,r) \in P(a)}
\]
\[
\textsf{C\textsubscript{tr}-$\vartriangleleft$}\;\frac{
\begin{aligned}
& P(x) \; \mathsf{prop} \; [x \in A] \\
& q_1(x,y) \in P(x) \; [x \in A, y \in x \,\varepsilon\,V] \\
& q_2(x,y,w) \in P(x) \; [x \in A, y \in I(x), w \in (\forall z \,\varepsilon\, C(x,y))P(z)] \\
& a \in A \quad i \in I(a) \quad p \in (\forall x \,\varepsilon\, C(a,i))x \vartriangleleft_{I,C}V
\end{aligned}}
{\mathsf{El}_\vartriangleleft(a,\mathsf{tr}(a,i,p),q_1,q_2) = q_2(a,i,\lambda z.\lambda w. \mathsf{El}_\vartriangleleft(z,p(z,w),q_1,q_2) \in P(a)}
\]

\paragraph*{Rules for coinductive positivity relation in iMF}
\[
\textsf{F-$\ltimes$}\;\frac{a \in A}{a \ltimes_{I,C} V \; \mathsf{prop_s}}
\]
\[
\textsf{E\textsubscript{corf}-$\ltimes$}\;\frac{
a \in A \quad p \in a \ltimes_{I,C} V
}{\mathsf{corf}(a,p) \in a \,\varepsilon\, V}
\]
\[
\textsf{E\textsubscript{cotr}-$\ltimes$}\;\frac{
a \in A \quad i \in I(a) \quad p \in a \ltimes_{I,C} V}{\mathsf{cotr}(a,i,p) \in (\exists x \,\varepsilon\, C(a,i))x \ltimes_{I,C} V}\]
\[
\textsf{I-$\ltimes$}\;\frac{\begin{aligned}
& P(x) \; \mathsf{prop} \; [x\in A]\\
& q_1(x,y) \in x \,\varepsilon\, V \; [x \in A,y \in P(x)]\\
& q_2(x,y,w) \in (\exists z \,\varepsilon\, C(x,y))P(z) \; [x \in A,y \in I(x),w\in P(x)]\\
& a \in A \quad p \in P(a)
\end{aligned}}
{\mathsf{coind}(a,p,q_1,q_2) \in a \ltimes_{I,C} V}
\]

The following result shows that the two flavours of (co)induction considered so far are equivalent.
\begin{theorem}
In both levels of the Minimalist Foundation, inductive basic covers and inductive predicates are mutually encodable, and so are coinductive positivity relations and coinductive predicates. In particular, coinductive predicates coincide with positivity predicates.
\end{theorem}
\begin{proof}
We first prove it for the extensional level. To see that topological (co)induction encodes (co)inductive predicates, it is enough to choose $V$ to be nilpotent; indeed, notice that, for each axiom set $(I,C)$ over $A$, the operator $\mathsf{Der}_{I,C}$ (resp. $\mathsf{Conf}_{I,C}$) is equivalent to the operator $\mathsf{Der}_{I,C,\emptyset}$ (resp. $\mathsf{Conf}_{I,C,A}$). It is trivial to check that the following interpretations satisfy the rules of (co)inductive predicates.
\begin{align*}
\mathsf{Ind}_{I,C}(x) & :\equiv x \vartriangleleft_{I,C} \emptyset \\
\mathsf{CoInd}_{I,C}(x) & :\equiv \mathsf{Pos}(x) \equiv x \ltimes_{I,C} A
\end{align*}

On the other hand, assume to have an axiom set $(I,C)$ over $A$, and a subset $V \in \mathcal{P}(A)$; to prove that inductive predicates encode basic inductive covers we define an enlarged rule set by encoding the additional reflexivity clause given by $V$. 
\begin{align*}
I_V(x) & :\equiv x \,\varepsilon\, V + I(x) \\
C_V(x,\mathsf{inl}(p)) & :\equiv \emptyset \\
C_V(x,\mathsf{inr}(y)) & :\equiv C(x,y)
\end{align*}
where, formally, we set $C_V(x,y,z)  :\equiv \mathsf{T}(\mathsf{El_+}(y,(w).[\bot],(w).[C(x,w,z)]))$. It is easy to check that the operator $\mathsf{Der}_{I,C,V}$ is equivalent to the operator $\mathsf{Der}_{I_V,C_V}$, thus making the following interpretation satisfy the rules of inductive basic covers.
\[
x \vartriangleleft_{I,C} V :\equiv\mathsf{Ind}_{I_V,C_V}(x)
\]
To prove that coinductive predicates encode coinductive positivity relations, we define a restriction of both the set $A$ and the axiom set $(I,C)$ by comprehension on those elements for which $V$ already holds.
\begin{align*}
A^V & :\equiv (\Sigma w \in A)w\,\varepsilon\, V \\
I^V(x) & :\equiv I(\pi_1(x)) \\
C^V(x,y) & :\equiv \{ z \in A^V \,|\, \pi_1(z) \,\varepsilon\, C(\pi_1(x),y) \}
\end{align*}
It is then easy to check that the following interpretation satisfies the rules of coinductive positivity relations as presented in Lemma \ref{altin}.
\[
x \ltimes_{I,C} V :\equiv (\exists p \in x \,\varepsilon\, V)\mathsf{CoInd}_{I^V,C^V}(\langle x , p \rangle)
\]
To adapt the above proofs for the intensional level, it is enough to explicitly interpret the proof terms and, in the inductive case, to check that the computation rules are satisfied. We only report the explicit interpretation of proof terms of inductive basic covers encoded by inductive predicates.
\begin{align*}
\mathsf{rf}(a,r) & :\equiv \mathsf{ind}(a,\mathsf{inl}(r),\lambda x.\lambda y.\mathsf{El}_\bot(y)) \\
\mathsf{tr}(a,i,p) & :\equiv \mathsf{ind}(a,\mathsf{inr}(i),p) \\
\mathsf{El}_\vartriangleleft(a,p,q_1,q_2) & :\equiv \mathsf{El_{Ind}}(a,p,(x,y,w).\mathsf{Ap}_\Rightarrow(f(x,y),w))
\end{align*}
where $f(x,y) :\equiv \mathsf{El}_+(y,(u).\lambda w.q_1(x,u),(u).\lambda w.q_2(x,u,w))$.
\end{proof}

As a corollary, we obtain that both the quotient model and the realizability interpretation of the Minimalist Foundation extend to (co)inductive predicates. We keep calling $\mathbf{MF}_{\mathsf{ind}}$, as in \cite{mmr21}, the (two-level) theory obtained by extending the Minimalist Foundation with inductive predicates and $\mathbf{MF}_{\mathsf{cind}}$, as in \cite{mmr22}, that obtained by extension with inductive and coinductive predicates.

\section{(Co)Induction in Martin-Löf's type theory}\label{hottsec}
Induction and coinduction in Martin-Löf's type theory and its extensions can assume many forms; two of the most common schemes are given by a pair of dual constructions called $\mathsf{W}$-types (or wellfounded trees) for induction and $\mathsf{M}$-types (or non-wellfounded trees) for coinduction. Here, we examine their relationship with the forms of (co)induction introduced in the previous section.

\subsection{Induction in Martin-Löf's type theory}
In Martin-Löf's type theory, inductive sets are usually considered according to the general scheme of \textit{$\mathsf{W}$-types}, also known as \textit{wellfounded trees}.
The parameters of a $\mathsf{W}$-type consist of a set $A$ and an $A$-indexed family of sets $B$, which together are often referred to as a \textit{container} \cite{indexed}; the resulting type $\mathsf{W}_{A,B}$ is intuitively understood as the set of wellfounded trees with nodes labelled by elements of $A$ and with a (possibly infinitary) branching function given by $B$. The precise rules of the constructor are reported below.
\paragraph*{Rules for $\mathsf{W}$-types in $\mathbf{MLTT_0}$}
\label{wMLTT}
\[
\textsf{ F-}\textsf{W}\;\frac{}{\mathsf{W}_{A,B} \in \mathsf{U_0}}
\]
\[
\textsf{ I-}\textsf{W}\;\frac
{a \in A \qquad f \in B(a) \to \mathsf{W}_{A,B}}
{\mathsf{sup}(a,f) \in \mathsf{W}_{A,B}}
\]
\[
\textsf{ E-}\textsf{W}\;
\frac{
\begin{aligned}
& M(w) \; \mathsf{type} \; [w \in \mathsf{W}_{A,B}]
\\
& d(x,h,k) \in M(\mathsf{sup}(x,h)) \; [x \in A \,, h \in B(x) \to \mathsf{W}_{A,B} \,, k \in (\Pi y \in B(x))M(h(y))]
\\
& t \in \mathsf{W}_{A,B}
\end{aligned}}{\mathsf{El_W}(t,(x,h,k).d) \in M(t)}
\]
\[
\textsf{ C-}\textsf{W}\;
\frac{
\begin{aligned}
& M(w) \; \mathsf{type} \; [w \in \mathsf{W}_{A,B}]
\\
& d(x,h,k) \in M(\mathsf{sup}(a,f)) \; [x \in A \,, h \in B(a) \to \mathsf{W}_{A,B} \,, k \in (\Pi y \in B(x))M(h(y))]
\\
& a \in A \quad f \in B(a) \to \mathsf{W}_{A,B}
\end{aligned}}
{\mathsf{El_W}(\mathsf{sup}(a,f),(x,h,k).d) = d(a,f,\lambda y.{\mathsf{El_W}(f(y),(x,h,k).d)) \in M(\mathsf{sup}(a,f))}}
\]

As far as we are concerned, however, $\mathsf{W}$-types do not seem to have the same expressive power of inductive predicates. This is because they produce just plain sets, while predicates, under the propositions-as-types paradigm, are families of sets. This same limitation has been addressed in \cite{Petersson}, where the authors proposed a generalization of $\mathsf{W}$-types, called \textit{dependent $\mathsf{W}$-types}, capable of constructing \textit{families of mutually} inductive sets; in the literature they are also known as \textit{general trees}, or \textit{indexed $\mathsf{W}$-types}. Dependent $\mathsf{W}$-types are interpreted again as sets of wellfounded trees with labelled nodes; however, each label now has a set of possible options for the branching function; moreover, each branching function not only indicates the number of subtrees but also dictates how each of their roots is to be labelled. The formal rendering of this intuition in Martin-Löf's type theory goes as follows. The parameters of a dependent $\mathsf{W}$-type, which, analogously to the non-dependent case, are referred to as \textit{indexed container}, consist of a small type $A \in \mathsf{U_0}$ of nodes' labels and a family of sets $I \in A \to  \mathsf{U_0}$ indexing the possible branching functions associated with each label. To formalise the branching functions, there are two possibilities in the literature.
Either with a function
\[
C(x,y) \in A \to \mathsf{U_0} \; [x \in A,y\in I(x)]
\]
that for each label says how many subtrees there are with roots labelled by it; or with two ariety functions
\begin{align*}
Br(x,y) \in \mathsf{U_0} & \; [x \in A,y\in I(x)] \\
ar(x,y) \in Br(x,y) \to A & \; [x \in A,y\in I(x)]
\end{align*}
that say how many subtrees there are in general and the label of each subtree's root, respectively. It is clear how, in the first case, indexed containers correspond precisely to the notion of rule set formulated in Martin-Löf's type theory. In either case, the type constructors are then formalised by adapting the pattern of $\mathsf{W}$-types to account for the extra indexing. In particular, using the first formulation, one obtains a proof-relevant version of inductive predicates in Martin-Löf's type theory; as always, their elimination rules differ from the ones in the Minimalist Foundation since now they can work towards sets depending on the constructor. For this reason, in the following discussion, we will refer to dependent $\mathsf{W}$-types defined with the parameter $C$ as \textit{(proof-relevant) inductive predicates}, reserving the name \textit{dependent $\mathsf{W}$-types} to just those defined using the parameters $Br$ and $ar$. Their precise rules are spelt out in the appendix (see \ref{dwMLTT} and \ref{indMLTT}). 
Finally, once more tracing on the pattern of dependent $\mathsf{W}$-types, in \cite{mmr21} the authors primitively introduced in Martin-Löf's type theory also a constructor $\vartriangleleft$ formalising proof-relevant inductive basic covers, whose rules are again recalled in the appendix \ref{indMLTT}.

By the results of \cite{topcount}, we know that all the inductive constructors presented so far can be reduced to $\mathsf{W}$-types.

\begin{theorem}\label{allind}
In $\mathbf{MLTT_0}$, the following type constructors are mutually encodable.
\begin{enumerate}
\item $\mathsf{W}$-types;
\item dependent $\mathsf{W}$-types;
\item inductive predicates;
\item inductive basic covers.
\end{enumerate}
\end{theorem}
\begin{proof}
Theorem 6.1 of \cite{topcount}.
\end{proof}

In the presence of function extensionality, the above constructors enjoy neat categorical semantics, which we should review since it will be vital next for treating coinduction.

The type $\mathsf{W}_{A,B}$ can be shown to be the support of an initial algebra for the so-called \textit{polynomial endofunctor} $\mathsf{P}_{A,B}$ on the category of types
\[
\mathsf{P}_{A,B}(X) :\equiv (\Sigma x \in A)(B(x) \to X)
\]
Analogously, inductive predicates, inductive basic covers and dependent $\mathsf{W}$-types enjoy categorical semantics as initial algebras of the following \textit{dependent polynomial endofunctors} on the category of $A$-dependent types, respectively.
\begin{align*}
\mathsf{Der}_{I,C}(P)(x) & :\equiv (\Sigma y \in I(x))(\Pi z \in A)(C(x,y,z) \to P(z)) \\
\mathsf{Der}_{I,C,V}(P)(x) & :\equiv V(x) + (\Sigma y \in I(x))(\Pi z \in A)(C(x,y,z) \to P(z)) \\
\mathsf{Der}_{Br,ar}(P)(x) & :\equiv (\Sigma y \in I(x))(\Pi z \in Br(x,y))P(ar(x,y,z))
\end{align*}
We chose to reuse the name $\mathsf{Der}$ because the first endofunctor clearly is the interpretation under the propositions-as-types paradigm of the derivability constructor used to define inductive predicates in the Minimalist Foundation.

We recall that polynomial endofunctors of the form $\mathsf{Der}_{Br,ar}$ can be described in an arbitrary locally cartesian closed category $\mathcal{C}$ \cite{w-typelccc} as follows. Indexed containers are specified by diagrams of the form
\[
A \xleftarrow{ar} Br \xrightarrow{br} I \xrightarrow{i} A
\]
and the resulting endofunctor is obtained as the composition of the following functors
\[
\mathcal{C}/A \xrightarrow{ar^*} \mathcal{C}/Br \xrightarrow{\Pi_{br}} \mathcal{C}/I \xrightarrow{\Sigma_i} \mathcal{C}/A
\]
where $ar^*$ is the pullback functor, $\Pi_{br}$ the dependent product functor given by the locally cartesian closure of $\mathcal{C}$, and $\Sigma_i$ is the coproduct functor, that is postcomposition by $i$.

The following lemma formally proves that, in the presence of function extensionality, the two choices we discussed for formalising the intuition behind a dependent well-founded tree are equivalent.

\begin{lemma}\label{eqder}
Each functor of the form $\mathsf{Der}_{I,C}$ is isomorphic to one of the form $\mathsf{Der}_{Br,ar}$; assuming $\mathsf{funext}$, also the converse holds.
\end{lemma}
\begin{proof}
To prove the first statement assume a rule set $(I,C)$ over a set $A$ and consider the following parameters.
\begin{align*}
Br(x,y) & :\equiv (\Sigma z \in A)C(x,y,z) \\
ar(x,y,w) & :\equiv \pi_1(w)
\end{align*}
On the other hand, assuming the parameters $A$, $I$, $Br$ and $ar$, consider the axiom set $(I,C)$, where
\[
C(x,y,z) :\equiv (\Sigma w \in Br(x,y))ar(x,y,w) =_A z
\]\end{proof}

Finally, notice that if we were interested just in the proof-irrelevant, logical semantics of those inductive constructors, we could regard the endofunctors above simply as endomorphism on the preorder reflection of the category of $A$-dependent sets and the type constructors just as their smallest fixed points.

\subsection{Coinduction in Martin-Löf's type theory}
As mentioned, in Martin-Löf type theory, coinduction is usually considered through a construction dual to $\mathsf{W}$-types, called \textit{$\mathsf{M}$-types} (also known as \textit{non-wellfounded trees}). The parameters of $\mathsf{M}$-types are the same as those of $\mathsf{W}$-types; as their alternative name suggests, a type $\mathsf{M}_{A,B}$ is intuitively understood as the set of non(-necessarily)-wellfounded trees with nodes labelled by elements of $A$ and with branching function given by $B$. Due to the lack of computational meaning for coinduction, $\mathsf{M}$-types are not usually presented through explicit inference rules; instead, they are characterised semantically by a universal property dual to that of $\mathsf{W}$-types, namely by being terminal coalgebras for polynomial endofunctors $\mathsf{P}_{A,B}$. Analogously, there also exist \textit{dependent $\mathsf{M}$-types}, dualising dependent $\mathsf{W}$-types and defined as the terminal coalgebras of dependent polynomial endofunctors $\mathsf{Der}_{Br,ar}$ or, equivalently by Lemma \ref{eqder}, $\mathsf{Der}_{I,C}$.
In these forms, plain and dependent $\mathsf{M}$-types have been shown to exist in Martin-Löf's type theory extended with function extensionality and axiom $\mathsf{K}$ \cite{coinMLTT}, and in Homotopy Type Theory \cite{m-types}.

Here, to define coinductive predicates and positivity relations in Martin-Löf's type theory, we follow instead the axiomatic approach taken in \cite{mmr22}, where rules for proof-relevant coinductive positivity relations have been explicitly introduced in Martin-Löf's type theory by just taking the corresponding rules for the intensional level of the Minimalist Foundation after identifying propositions with sets.

\paragraph*{Rules for coinductive predicates in $\mathbf{MLTT_0}$}
\[
\textsf{F-$\mathsf{CoInd}$}\;\frac{}{\mathsf{CoInd}_{I,C} \in A \to \mathsf{U_0}}
\]
\[
\textsf{E-$\mathsf{CoInd}$}\;\frac{a \in A \quad i \in I(a) \quad p \in \mathsf{CoInd}_{I,C}(a)}{ \mathsf{des}(a,i,p) \in (\Sigma x \in A)(C(a,i,x) \times \mathsf{CoInd}_{I,C}(x))}
\]
\[
\textsf{I-$\mathsf{CoInd}$}\;\frac{
\begin{aligned}
& M(x) \; \mathsf{type} \; [x \in A] \\
& d(x,y,w) \in (\Sigma z \in A)(C(x,y,z) \times M(z)) \; [x \in A, y \in I(x), w \in P(x)] \\
& a \in A \quad m \in M(a)
\end{aligned}}{\mathsf{coind}(a,m,(x,y,w).d) \in \mathsf{CoInd}_{I,C}(a)}
\]

\paragraph*{Rules for coinductive positivity relations in $\mathbf{MLTT_0}$}
\[
\textsf{F-$\ltimes$}\;\frac{}{- \ltimes_{I,C} V \in A \to \mathsf{U_0}}
\]
\[
\textsf{E-corf-$\ltimes$}\;\frac{
a \in A \quad p \in a \ltimes_{I,C} V
}{\mathsf{corf}(a,p) \in V(a)}
\]
\[
\textsf{E-cotr-$\ltimes$}\;\frac{
a \in A \quad i \in I(a) \quad p \in a \ltimes_{I,C} V}{\mathsf{cotr}(a,i,p) \in (\Sigma x \in A)(C(a,i,x) \times x \ltimes_{I,C} V)}\]
\[
\textsf{I-$\ltimes$}\;\frac{\begin{aligned}
& M(x) \; \mathsf{type} \; [x \in A]\\
& q_1(x,y) \in V(x) \; [x \in A,y \in M(x)]\\
& q_2(x,y,w) \in (\Sigma z \in A)(C(x,y,z) \times x \ltimes_{I,C} V) \; [x \in A,y \in I(x),w\in M(x)]\\
& a \in A \quad p \in M(a)
\end{aligned}}
{\mathsf{coind}(a,p,q_1,q_2) \in a \ltimes_{I,C} V}
\]
Notice that in this way, coinductive predicates and positivity relations are interpreted as the greatest fixed point of the following \textit{dependent
copolynomial endofunctors} defined in Martin-Löf's type theory and viewed as endomorphisms on the preorder reflection of the category of $A$-dependent types, respectively.
\begin{align*}
\mathsf{Conf}_{I,C}(P)(x) &:\equiv (\Pi y \in I(x))(\Sigma z \in A)(C(x,y,z) \times P(z)) \\
\mathsf{Conf}_{I,C,V}(P)(x) & :\equiv V(x) \times (\Pi y \in I(x))(\Sigma z \in A)(C(x,y,z) \times P(z))
\end{align*}
Clearly, we also define, analogously to the inductive case, an alternative copolynomial endofunctor $\mathsf{Conf}_{Br,ar}$ induced by a indexed container $(A,I,Br,ar)$ as
\[
\mathsf{Conf}_{Br,ar}(P)(x) :\equiv (\Pi y \in I(x))(\Sigma z \in Br(x,y))P(ar(x,y,z))
\]
We can derive the dual result relating the two versions.
\begin{lemma}\label{eqder2}
Each functor of the form $\mathsf{Conf}_{I,C}$ is isomorphic to one of the form $\mathsf{Conf}_{Br,ar}$; assuming $\mathsf{funext}$, also the converse holds.
\end{lemma}
\begin{proof}
The proof is identical to that of Lemma \ref{eqder}.
\end{proof}
Finally, we can give a purely categorical description also of the endofunctors $\mathsf{Con}_{Br,ar}$ in an arbitrary locally cartesian closed category $\mathcal{C}$ as the composition of the following functors
\[
\mathcal{C}/A \xrightarrow{ar^*} \mathcal{C}/Br \xrightarrow{\Sigma_{br}} \mathcal{C}/I \xrightarrow{\Pi_i} \mathcal{C}/A
\]

By how they are defined, coinductive predicates do not have with inductive ones the same relation that $\mathsf{M}$-types have with $\mathsf{W}$-types: while coinductive predicates are defined as the greatest fixed points of the operator $\mathsf{Conf}$, which is itself dual to the defining operator $\mathsf{Der}$ for inductive predicates, $\mathsf{M}$-types and $\mathsf{W}$-types (and their dependent versions) are terminal coalgebras and initial algebras, respectively, of the \textit{same} endofunctors. For this reason, we cannot hope to get a fully symmetrical result to the one obtained in the inductive case. Nonetheless,  in Martin-Löf's type theory, thanks to the axiom of choice, we can prove that the class of constructors $\mathsf{Der}_{I,C}$ is far more expressive than its dual $\mathsf{Conf}_{I,C}$, and it subsumes them.

\begin{proposition}\label{last}
Each functor of the form $\mathsf{Conf}_{I,C}$ is isomorphic to one of the form
$\mathsf{Der}_{Br,ar}$.
\end{proposition}
\begin{proof}
Assume to have a rule set $(I,C)$ over $A$. It is easy to check that the functor $\mathsf{Conf}_{I,C}$ is isomorphic to the functor $\mathsf{Der}$ constructed using the following parameters.
\begin{align*}
I'(x) & :\equiv (\Pi y \in I(x))(\Sigma z \in A)C(x,y,z) \\
Br(x,f) & :\equiv I(x) \\
ar(x,f,y) & :\equiv \pi_1(f(y))
\end{align*}
\end{proof}

\begin{corollary}\label{allco}
Coinductive predicates and coinductive positivity relations are mutually encodable in $\mathbf{MLTT_0}$. Moreover, they are both encodable in any theory extending $\mathbf{MLTT_0}$ in which the greatest fixed points of operators $\mathsf{Der}_{Br,ar}$ exist.
\end{corollary}
\begin{proof}
Since the rules of the constructors are identical, the proof of the first statement is entirely analogous to that of the intensional level of the Minimalist Foundation. Then, it is enough to prove the second statement for coinductive predicates. If the theory admits a greatest fixed point for every endofunctor of the form $\mathsf{Der}_{Br,ar}$, then, by Proposition \ref{last}, it equivalently admits one also for every endofunctor of the form $\mathsf{Conf}_{I,C}$; the former can then be used to interpret the coinductive predicates $\mathsf{CoInd}_{I,C}(x)$.
\end{proof}

Since the hypotheses of the above proposition only require the existence of a greatest fixed point, it follows as an immediate corollary that any extension of $\mathbf{MLTT_0}$ in which dependent $\mathsf{M}$-types exist encodes coinductive predicates.

Finally, notice that if we were to take the alternative route of semantically defining coinductive predicates as terminal coalgebras of the endofunctors $\mathsf{Conf}_{I,C}$, Proposition \ref{last} would have implied that they are precisely a subclass of $\mathsf{M}$-types.

\section{Compatibility of (co)induction}\label{compsec}
The results obtained in the previous sections allow us to extend some of the compatibility results of \ref{comp} to (co)inductive definitions.

\begin{corollary}\label{coro}
\begin{enumerate}
\item $\mathbf{iMF}_\mathsf{ind}$ is compatible with $\mathbf{MLTT_0}+\mathsf{W}$;
\item $\mathbf{iMF}_\mathsf{ind}$ is compatible with $\mathbf{HoTT}$;
\item $\mathbf{iMF}_\mathsf{cind}$ is compatible with $\mathbf{MLTT_0}+\mathsf{W}+\mathsf{funext}+\mathsf{K}$;
\item $\mathbf{iMF}_\mathsf{cind}$ is compatible with $\mathbf{CIC}$;
\item $\mathbf{eMF}_\mathsf{cind}$ is compatible with $\mathcal{T}_{\mathbf{Topos}}$.
\item $\mathbf{eMF}_\mathsf{ind}$ is compatible with $\mathbf{CZF}+\mathsf{REA}$;
\item $\mathbf{eMF}_\mathsf{cind}$ is compatible with $\mathbf{CZF}+\mathsf{RRS}\text{-}\bigcup\mathsf{REA}$.
\end{enumerate}

In the last two points above, $\mathsf{REA}$ and $\mathsf{RRS}\text{-}\bigcup\mathsf{REA}$ are the \textit{Extension Axiom schemes} defined for Regular sets and Strongly Regular sets satisfying the Relation Reflection Scheme, respectively. They were introduced to accommodate inductive and coinductive definitions in constructive set theory; for their precise statements see \cite{czf}.
\end{corollary}
\begin{proof}
\begin{enumerate}
\item Inductive predicates can be interpreted straightforwardly into their analogues defined in Martin-Löf's type theory. In turn, thanks to Theorem \ref{allind}, we know how to construct the latter using $\mathsf{W}$-types.
\item Despite $\mathbf{HoTT}$ being an extension of $\mathbf{MLTT_0}+\mathsf{W}$, the interpretation defined in the previous point ceases to be compatible once the target theory is changed to $\mathbf{HoTT}$. This is because propositions of the Minimalist Foundation should be interpreted as h-propositions. The solution is to interpret inductive predicates using the expressive power of Higher Inductive Types, as already observed in the case of inductive basic covers in \cite{ayb}, by following a standard trick which postulates, besides the introduction constructor $\mathsf{ind}$, another constructor whose function is to trivialise the identity type.
\item It is enough to extend the interpretation described in the first point by interpreting $\mathbf{iMF}$-coinductive predicates as $\mathbf{MLTT_0}$-coinductive predicates, once we know that the target theory supports them since by the results in \cite{indexed} it satisfies the hypotheses of Proposition \ref{allco}.
\item By Proposition \ref{iimp}.
\item By Remark \ref{imp}.
\item By Theorem 4.6 of \cite{mmr22}.
\item By Theorem 13.2.3 of \cite{czf}.
\end{enumerate}    
\end{proof}

\begin{remark}
In \cite{mmr22}, the authors proved a compatibility result for coinductive predicates in Martin-Löf's type theory alternative to the one in the second point of the above Corollary. There, instead of constructing coinductive predicates as $\mathsf{M}$-types, they assumed in the target theory the existence of a Palmgren's superuniverse and encoded them directly.
\end{remark}

\begin{remark}
Notice that the target theory $\mathbf{MLTT_0}+\mathsf{W}+\mathsf{funext}+\mathsf{K}$ of the third point of the above corollary can be interpreted in the h-set-theoretic fragment of $\mathbf{HoTT}$. However, the interpretation used there is not a compatible interpretation of $\mathbf{iMF}_\mathsf{cind}$ in $\mathbf{HoTT}$ because $\mathbf{MLTT_0}$-coinductive predicates are not necessarily h-propositions. Thus, compatibility remains an open problem since it is not evident how to generate coinductive h-propositions in $\mathbf{HoTT}$.
\end{remark}

\section{Conclusion and future work}
We have shown that (co)inductive methods of formal topology are equivalent to (co)inductive predicates and that, in turn, they can all be constructed in Martin-Löf's type theory using (non)-wellfounded trees.

In the future, we aim to motivate the addition of $\mathsf{M}$-types in the Minimalist Foundation and implement them in the quotient model and the realizability interpretation. Further goals would be to deepen the categorical notion of copolynomial functors in a locally cartesian closed category.

\subsection*{Acknowledgments}
We thank Francesco Ciraulo, Milly Maietti, and Giovanni Sambin for useful discussions on the subject of this paper.

\printbibliography

\appendix

\section{Rules}

In the following, when we give the rules of a type constructor, we interpret the formation rule's premises as parameters of the constructor; we then take them for granted in the premises of the other constructor's rules.

\paragraph*{Rules for dependent wellfounded trees in $\mathbf{MLTT_0}$}
\label{dwMLTT}
\[
\frac{\begin{aligned}
& A  \in \mathsf{U_0} \\
& I(x) \in \mathsf{U_0} \; [x \in A] \\
& Br(x,y) \in \mathsf{U_0} \; [x \in A \,, y \in I(x)] \\
& ar(x,y) \in Br(x,y) \to A \; [x \in A \,, y \in I(x)]
\end{aligned}
}{\mathsf{DW}_{Br,ar} \in A \to \mathsf{U_0}}
\textsf{ F-}\textsf{DW}
\]
\[
\frac{\displaystyle
a \in A \qquad i \in I(a) \qquad f \in (\Pi z \in Br(a,i))\mathsf{DW}_{Br,ar}(ar(a,i,z))
}{\mathsf{dsup}(a,i,f) \in \mathsf{DW}_{Br,ar}(a)}\textsf{ I-}\textsf{DW}
\]
\[
\frac{
\begin{aligned}
& M(x,w) \; \mathsf{type} \; [x \in A \,, w \in \mathsf{DW}_{Br,ar}(x)] \\
& d(x,y,h,k) \in M(x,\mathsf{dsup}(x,y,h)) \; \\
& \quad [x \in A, y \in I(x), \\
& \quad\quad h \in (\Pi z \in Br(x,y))\mathsf{DW}_{Br,ar}(ar(x,y,z)), \\
& \quad\quad\quad k \in (\Pi z \in Br(x,y))M(ar(x,y,z),h(z))] \\
& a \in A \qquad t \in \mathsf{DW}_{Br,ar}(a)
\end{aligned}}
{\mathsf{El_{DW}}(a,t,(x,y,h,k).d) \in M(a,t)}
\textsf{ E-}\textsf{DW}
\]
\[
\frac{
\begin{aligned}
& M(x,w) \; \mathsf{type} \; [x \in A \,, w \in \mathsf{DW}_{Br,ar}(x)] \\
& d(x,y,h,k) \in M(x,\mathsf{dsup}(x,y,h)) \; \\
& \quad [x \in A, y \in I(x), \\
& \quad\quad h \in (\Pi z \in Br(x,y))\mathsf{DW}_{Br,ar}(ar(x,y,z)), \\
& \quad\quad\quad k \in (\Pi z \in Br(x,y))M(ar(x,y,z),h(z))] \\
& a \in A \quad i \in I(a) \quad f \in (\Pi z \in Br(a,i))\mathsf{DW}_{Br,ar}(ar(a,i,z))
\end{aligned}}
{\mathsf{El_{DW}}(a,\mathsf{dsup}(a,i,f),d) = d(a,i,f,\lambda z.\mathsf{El_{DW}}(ar(a,i,z),f(z),d)) \in M(a,\mathsf{dsup}(a,i,f))}
\textsf{ C-}\textsf{DW}
\]

\paragraph*{Rules for inductive predicates in $\mathbf{MLTT_0}$}
\label{indMLTT}
\[
\frac{\begin{aligned}
& A  \in \mathsf{U_0} \\
& I(x) \in \mathsf{U_0} \; [x \in A] \\
& C(x,y) \in A \to \mathsf{U_0} \; [x \in A \,, y \in I(x)] \\
\end{aligned}
}{\mathsf{Ind}_{I,C} \in A \to \mathsf{U_0}}
\textsf{ F-}\textsf{Ind}
\]
\[
\frac{\displaystyle
a \in A \qquad i \in I(a) \qquad p \in (\Pi x \in A)(C(a,i,x) \to \mathsf{Ind}_{I,C}(x))
}{\mathsf{ind}(a,i,p) \in \mathsf{Ind}_{I,C}(a)}\textsf{ I-}\textsf{Ind}
\]
\[
\frac{
\begin{aligned}
& M(x,w) \; \mathsf{type} \; [x \in A \,, w \in \mathsf{Ind}_{I,C}(x)] \\
& d(x,y,h,k) \in M(x,\mathsf{ind}(x,y,h)) \; \\
& \quad [x \in A, y \in I(x), \\
& \quad\quad h \in (\Pi z \in A)(C(x,y,z) \to \mathsf{Ind}_{I,C}(x)), \\
& \quad\quad\quad k \in (\Pi z \in A)(\Pi q \in C(x,y,z))M(z,h(z,q))] \\
& a \in A \qquad p \in \mathsf{Ind}_{I,C}(a)
\end{aligned}}
{\mathsf{El_{Ind}}(a,p,(x,y,h,k).d) \in M(a,p)}
\textsf{ E-}\textsf{Ind}
\]
\[
\frac{
\begin{aligned}
& M(x,w) \; \mathsf{type} \; [x \in A \,, w \in \mathsf{Ind}_{I,C}(x)] \\
& d(x,y,h,k) \in M(x,\mathsf{ind}(x,y,h)) \; \\
& \quad [x \in A, y \in I(x), \\
& \quad\quad h \in (\Pi z \in A)(C(x,y,z) \to \mathsf{Ind}_{I,C}(x)), \\
& \quad\quad\quad k \in (\Pi z \in A)(\Pi q \in C(x,y,z))M(z,h(z,q))] \\
& a \in A \qquad i \in I(a) \qquad p \in (\Pi x \in A)(C(a,i,x) \to \mathsf{Ind}_{I,C}(x))
\end{aligned}}
{\mathsf{El_{Ind}}(a,\mathsf{ind}(a,i,p),(x,y,h,k).d) = d(a,i,p,\lambda z.\lambda q.\mathsf{El_{Ind}}(z,p(z,q),d)) \in M(a,\mathsf{ind}(a,i,p))}
\textsf{ C-}\textsf{Ind}
\]

\paragraph*{Rules for inductive basic covers in $\mathbf{MLTT_0}$}
\label{ibcMLTT}

\[
\frac{\begin{aligned}
& A \in \mathsf{U_0} \\
& I(x) \in \mathsf{U_0} \; [x \in A] \\
& C(x,y) \in A \to \mathsf{U_0} \; [x \in A,y \in I(x)] \\
& V \in A \to \mathsf{U_0}
\end{aligned}}{- \vartriangleleft_{I,C} V \in A \to \mathsf{U_0}}
\;\textsf{ F-}\textsf{$\vartriangleleft$}
\]
\[
\frac{a \in A \qquad r \in V(a)}
{\mathsf{rf}(a,r) \in a \vartriangleleft_{I,C} V}
\;\textsf{I\textsubscript{rf} -$\vartriangleleft$}
\]
\[
\frac{a \in A \qquad i \in I(a) \qquad r \in (\Pi x \in A)(C(a,i,x) \to x \vartriangleleft_{I,C} V)}
{\mathsf{tr}(a,i,r) \in a \vartriangleleft_{I,C} V}
\;\textsf{I\textsubscript{tr}-$\vartriangleleft$}
\]
\[
\frac{
\begin{aligned}
& M(x,w) \; \mathsf{type} \; [x \in A, w \in x \vartriangleleft_{I,C} V] \\
& q_1(x,y) \in M(x,\mathsf{rf}(x,y)) \; [x \in A \,, y \in V(a)] \\
& q_2(x,y,h,k) \in M(x,\mathsf{tr}(x,y,h))  \\
& \quad [x \in A\,, y \in I(x)\,, \\
& \quad\quad h \in (\Pi z \in A)(C(x,y,z) \to z \vartriangleleft_{I,C} V) \,, \\
& \quad\quad\quad k \in (\Pi z \in A)(\Pi q \in C(x, y, z))M(z,h(z,q)) ] \\
& a \in A \qquad p \in a \vartriangleleft_{I,C} V
   \end{aligned}}{\mathsf{El}_{\vartriangleleft}(a,p,(x,y).q_1,(x,y,h,k).q_2) \in M(a,p)}
\;\textsf{E -}\vartriangleleft
\]
\[
\frac{
\begin{aligned}
& M(x,w) \; \mathsf{type} \; [x \in A, w \in x \vartriangleleft_{I,C} V] \\
& q_1(x,y) \in M(x,\mathsf{rf}(x,y)) \; [x \in A \,, y \in V(a)] \\
& q_2(x,y,h,k) \in M(x,\mathsf{tr}(x,y,h))  \\
& \quad [x \in A\,, y \in I(x)\,, \\
& \quad\quad h \in (\Pi z \in A)(C(x,y,z) \to z \vartriangleleft_{I,C} V) \,, \\
& \quad\quad\quad k \in (\Pi z \in A)(\Pi q \in C(x, y, z))M(z,h(z,q)) ] \\
& a \in A \qquad r \in V(a)
   \end{aligned}}{\mathsf{El}_{\vartriangleleft}(a,\mathsf{rf}(a,r),q_1,q_2) = q_1(a,r) \in M(a,\mathsf{rf}(a,r))}
\;\textsf{C\textsubscript{rf} -}\vartriangleleft
\]
\[
\frac{
\begin{aligned}
& M(x,w) \; \mathsf{type} \; [x \in A, w \in x \vartriangleleft_{I,C} V] \\
& q_1(x,y) \in M(x,\mathsf{rf}(x,y)) \; [x \in A \,, y \in V(a)] \\
& q_2(x,y,h,k) \in M(x,\mathsf{tr}(x,y,h))  \\
& \quad [x \in A\,, y \in I(x)\,, \\
& \quad\quad h \in (\Pi z \in A)(C(x,y,z) \to z \vartriangleleft_{I,C} V) \,, \\
& \quad\quad\quad k \in (\Pi z \in A)(\Pi q \in C(x, y, z))M(z,h(z,q)) ] \\
& a \in A \qquad i \in I(a) \qquad r \in (\Pi x \in A)(C(a,i,x) \to x \vartriangleleft_{I,C} V)
   \end{aligned}}{\mathsf{El}_{\vartriangleleft}(a,\mathsf{tr}(a,i,r),q_1,q_2) = q_2(a,i,\lambda z.\lambda q. \mathsf{El}_{\vartriangleleft}(z,r(z,q),q_1,q_2)) \in M(a,\mathsf{tr}(a,i,r))}
\;\textsf{C\textsubscript{tr} -}\vartriangleleft
\]

\end{document}